\documentclass[11pt]{amsart}

\usepackage{amsmath, amsthm, amssymb, amsfonts, enumerate,stmaryrd,bbm}
\usepackage[colorlinks=true,linkcolor=blue,urlcolor=blue]{hyperref}
\usepackage{dsfont}
\usepackage{color}
\usepackage{geometry}
\usepackage{todonotes}
\usepackage{soul}
\usepackage{graphicx}
\usepackage{caption}
\usepackage{mathtools}

\usepackage{appendix}

\usepackage{tikz}
\tikzset{every node/.style={align=center}}

\mathtoolsset{showonlyrefs}

\makeatletter
\@namedef{subjclassname@2020}{%
  \textup{2020} Mathematics Subject Classification}
\makeatother

\newcommand\cC{\mathcal{C}}
\newcommand\cD{\mathcal{D}}

\newcommand\cF{\mathcal{F}}

\newcommand\cL{\mathcal{L}}
\newcommand\cM{\mathcal{M}}
\newcommand\cN{\mathcal{N}}
\newcommand\cO{\mathcal{O}}

\newcommand\cS{\mathcal{S}}
\newcommand\cT{\mathcal{T}}

\newcommand\bE{\mathbb{E}}
\newcommand\bF{\mathbb{F}}

\newcommand\bP{\mathbb{P}}

\newcommand\bR{\mathbb{R}}

\def \ud{\mathrm{d}}
\def \e{\mathrm{e}}

\newcommand{\eps}{\varepsilon}

\newcommand{\ind}{\mathbbm{1}}

\newtheorem{theorem}{Theorem}[section]
\newtheorem{lemma}[theorem]{Lemma}
\newtheorem{corollary}[theorem]{Corollary}
\newtheorem{proposition}[theorem]{Proposition}

\newtheorem{remark}[theorem]{Remark}
\newtheorem{assumptions}[theorem]{Assumption}

\title[A result for time-inhomogeneous optimal stopping problems]{On the monotonicity of the stopping boundary \\ for time-inhomogeneous optimal stopping problems}
\author[A.\ Milazzo]{Alessandro Milazzo}
\address{A.\ Milazzo: Department of Mathematics, Uppsala University, Box 480, 75106, Uppsala, SWEDEN.}
\email{\href{mailto:alessandro.milazzo@math.uu.se}{alessandro.milazzo@math.uu.se}}

\numberwithin{equation}{section}

\begin{document}
	
\begin{abstract}
We consider a class of time-inhomogeneous optimal stopping problems and we provide sufficient conditions on the data of the problem that guarantee monotonicity of the optimal stopping boundary. In our setting, time-inhomogeneity stems not only from the reward function but, in particular, from the time dependence of the drift coefficient of the one-dimensional stochastic differential equation (SDE) which drives the stopping problem. In order to obtain our results, we mostly employ probabilistic arguments: we use a comparison principle between solutions of the SDE computed at different starting times, and martingale methods of optimal stopping theory. We also show a variant of the main theorem, which weakens one of the assumptions and additionally relies on the renowned connection between optimal stopping and free-boundary problems.
\end{abstract}
	
\maketitle

\section{introdution}

In this paper we consider a general class of time-inhomogeneous optimal stopping problems and we provide simple sufficient conditions on the data of the problem that guarantee monotonicity of the optimal stopping boundary. The novelty of our work is to prove this result when the underlying process is time-inhomogeneous. In our setting, the underlying process is the unique strong solution of a one-dimensional stochastic differential equation (SDE) whose drift coefficient may be time-dependent. We first show how to obtain monotonicity of the optimal stopping boundary when the reward function is time-homogeneous and then we extend the result to the case of a time-dependent reward function, when it is sufficiently regular to apply Ito's formula. We focus our attention on finite-horizon optimal stopping problems but our methods clearly apply also to infinite-horizon optimal stopping problems, as the latter do not carry an additional time-dependence in the domain of the admissible stopping times.

The behaviour of the optimal stopping boundary $t\mapsto b(t)$ is crucial in order to fully characterise an optimal stopping problem. In particular, continuity and monotonicity of the map $t\mapsto b(t)$ are two desirable properties. However, this regularity is usually studied on a case-by-case basis and the number of works that provide sufficient conditions to obtain these results in a general framework is limited. Classical tricks to show continuity of the stopping boundary are presented in \cite{peskir2006optimal} in various examples, whereas results in a general setting can be found in \cite{de2015note} (for one-dimensional diffusions) and \cite{peskir2019continuity} (for two-dimensional diffusions). Determining monotonicity of $t\mapsto b(t)$ can be even a more relevant turning point. First, it is a helpful result in order to obtain its continuity (as shown, e.g., in \cite{de2015note}). Furthermore, when the underlying process is strong Markov, it implies that the optimal stopping time $\tau^*_{t,x}$ is a continuous function of the starting point $(t,x)$ across the boundary\footnote{Here, we mean that if $(t,x)=(t,b(t))$ and $(t_n,x_n)\to(t,x)$ as $n\to\infty$, then $\tau^*_{t_n,x_n}\to\tau^*_{t,x}$ as $n\to\infty$, $\bP$-a.s.} or, equivalently, that the boundary is regular for the interior of the stopping set in the sense of diffusions (a concept extensively illustrated in \cite{de2020global}). This yields global $C^1$-regularity of the value function, which is also a helpful result to characterise the stopping boundary (when continuous) as the unique continuous solution of a family of integral equations. An extensive probabilistic analysis of the geometry of a general class of optimal stopping problems, including continuity and monotonicity of the stopping boundary, is presented in \cite{de2022stopping} when the underlying diffusion and reward function are time-homogeneous. The shape of the continuation region is also studied under a general framework in \cite{jacka1992finite}. However, their result on the monotonicity of $t\mapsto b(t)$ (see Proposition 4.4 therein) holds only for time-homogeneous diffusions. One contribution of this paper is to extend this result to a class of time-inhomogeneous diffusions. Regularity and characterisation of the value function are obtained for time-inhomogeneous Markov processes in \cite{oshima2006optimal} and in \cite{yang2014refined}, and for time-inhomogeneous Poisson processes in \cite{hobson2021shape}. To the best of our knowledge, no study of the properties of the stopping boundary has been developed in a general setting for time-inhomogeneous diffusions. It is also worth mentioning several theoretical works on the behaviour of the stopping boundary and of the value function in the context American options. We cite, among others, \cite{chen2007mathematical}, \cite{ekstrom2004convexity}, \cite{jacka1991optimal}, \cite{jaillet1990variational}, \cite{laurence2009regularity} and \cite{villeneuve1999exercise}.

In order to obtain our results, we rely on probabilistic arguments. We first present a comparison principle between solutions of the underlying SDE computed at different starting times (see Lemma \ref{lemma:CompPrinc}). Specifically, we show that if the drift coefficient $t\mapsto\mu(t,x)$ is monotone then the solutions of the SDE computed at different starting times are ordered. By means of this result and martingale methods of optimal stopping theory, we prove that if in addition a time-homogeneous reward function $x\mapsto g(x)$ is non-decreasing then $t\mapsto v(t,x)$ is also monotone for every $x\in\bR$ (see Theorem \ref{thm:v_t}). In a variant of the theorem we show that if monotonicity of $t\mapsto\mu(t,x)$ does not hold for every $x$ in the state space of the underlying process, we are able to weaken this condition and obtain the same result under a further assumption which involves the derivatives of the value function and it is implied by convexity of $x\mapsto v(t,x)$ (see Theorem \ref{thm:v_tTauB}). This proof additionally relies on the renowned connection between optimal stopping and free-boundary problems. An example of time-inhomogeneous diffusions which perfectly fits the weaker monotonicity assumption (of Theorem \ref{thm:v_tTauB}) on $t\mapsto\mu(t,x)$ is given by Brownian bridges. Several works have investigated optimal stopping problems involving Brownian bridges and we cite, among others, \cite{shepp1969}, \cite{follmer1972}, \cite{ekstrom2009optimal}, \cite{ekstrom2020optimal}, \cite{de2020optimal}, \cite{glover2020optimally} and \cite{d2020discounted}. Both Theorem \ref{thm:v_t} and Theorem \ref{thm:v_tTauB} lead to the monotonicity of the optimal stopping boundary $t\mapsto b(t)$ (see Corollary \ref{cor:Monot}). Then, we prove that monotonicity of $t\mapsto b(t)$ can be obtained even when the reward function $g$ depends on time (see Theorem \ref{thm:MonotExt}). This extension holds when $g$ is sufficiently regular to apply Ito's formula and under the additional assumption of monotonicity of $t\mapsto \cL g (t,x)$, where $\cL$ denotes the infinitesimal generator of the underlying diffusion.

Our methods are particularly suited to study optimal stopping problems under incomplete information. The common feature of these problems is a random variable whose outcome is unknown to the optimiser and which affects the drift of the underlying process and/or the reward function. The literature is vast and diverse in this field and we cite, among others, \cite{shiryaev1967two}, \cite{decamps2005investment}, \cite{ekstrom2011optimal}, \cite{ekstrom2016optimal}, \cite{ekstrom2020optimal}, \cite{ekstrom2019american}, \cite{gapeev2012pricing}, \cite{glover2020optimally} \cite{henderson2020executive}. Our results apply, in particular, to models as in \cite{ekstrom2020optimal} and \cite{glover2020optimally} where a random variable affects the drift of the underlying process and, in a Bayesian formulation of the problem, only the prior distribution of the random variable is known to the optimiser. As time evolves, the information obtained from observing the underlying process is used to update the initial beliefs about the unknown random variable. By filtering theory, the underlying process can be expressed as a time-inhomogeneous diffusion whose time-dependent drift is the conditional expectation of the unknown random variable given the observations of the process, which can be obtained through the prior distribution. This, thus, fits into our framework, as we illustrate in Section \ref{sect:OptStoppPartInfo}.

The rest of the paper is organised as follows. In Section \ref{sect:PbForm} we formulate the starting problem and we recall some standard results on optimal stopping theory. In Section \ref{sect:CompPrinc} we provide a comparison principle between solutions of the underlying SDE starting at different times, which will be later used in Section \ref{sect:MainResult} to determine the monotonicity of the optimal stopping boundary. In Section \ref{sect:Extension} we extend the range of applicability for the results of Section \ref{sect:MainResult} by considering stopping problems where also the reward functions may depend on time. Our methods are particularly suited to study a class of optimal stopping problems under partial information, which we describe in Section \ref{sect:OptStoppPartInfo}. We conclude by illustrating, in Section \ref{sect:Examples}, some simple examples of optimal stopping problems where our results apply.

\section{Starting problem and background results}\label{sect:PbForm}

Let $(\Omega,\cF,\bP)$ be a complete probability space with a filtration $\bF:=(\cF_t)_{t\geq 0}$ satisfying the usual conditions and let $W:=(W_t)_{t\geq0}$ be a standard Brownian motion which is $\bF$-adapted. Let $T\in(0,\infty)$ be a finite time horizon. In this paper we treat finite-horizon optimal stopping problems, but it will be clear that our methods apply also to the infinite-horizon analogues, where the time-dependence of the value function stems only from the drift coefficient of the underlying SDE and not from the domain of the admissible stopping times.

Given an initial condition $X_t=x\in\bR$ for $t\in[0,T)$, let $X=(X_s)_{s\geq t}$ be the time-inhomogeneous stochastic process described by
\begin{equation}\label{eq:X}
X_{t+s}=x+\int_0^s \mu(t+r,X_{t+r})\ud r+\int_0^s \sigma(X_{t+r})\ud W_r, \qquad s\in[0,T-t],
\end{equation}
where $\mu:[0,T]\times\bR\to\bR$ and $\sigma:\bR\to\bR$ are, respectively, the drift and diffusion coefficients. We assume that $x\mapsto \mu(t,x)$ is Lipschitz-continuous for every $t\in[0,T]$ and that $x\mapsto\sigma(x)$ satisfies the standard Yamada-Watanabe condition which guarantees the strong existence and uniqueness of the solution for the SDE \eqref{eq:X} (see, e.g., \cite[Theorem 40.1]{rogers2000diffusions}). Namely, we assume that there exists an increasing function $h:[0,\infty)\to[0,\infty)$ such that
$$\int_0^\eps h^{-1}(s)\ud s=\infty, \qquad \forall\: \eps>0$$
and
\begin{equation}\label{eq:Assh}
\big(\sigma(x)-\sigma(y)\big)^2\leq h(|x-y|), \qquad \forall \: x,y\in\bR.
\end{equation}
In order to keep track of the initial condition $X_t=x$, we will sometimes denote the solution $X$ of the SDE \eqref{eq:X} by $X^{t,x}$.

Given a (terminal) reward function $g:\bR\to\bR$, we define the optimal stopping problem
\begin{equation}\label{eq:v}
v(t,x):=\sup_{\tau\in\cT_t}\bE\Big[g(X^{t,x}_{t+\tau})\Big], \qquad (t,x)\in[0,T]\times\bR,
\end{equation} 
where $\cT_t$ is the class of $\bF$-stopping times $\tau$ such that $\tau\in[0,T-t]$, $\bP$-a.s. To simplify the exposition, we start by considering stopping problems of the form \eqref{eq:v}. We then extend our results to stopping problems that include both a running reward function and a terminal reward function which may also depend on time (see Section \ref{sect:Extension}).

Let $\cC$ be the continuation region and its complement $\cD:=\cC^c$ be the stopping region, respectively, defined by
\begin{equation}\label{eq:cC}
\cC:=\{(t,x)\in[0,T]\times\bR: v(t,x)>g(x) \}
\end{equation}
and
\begin{equation}
\cD:=\{(t,x)\in[0,T]\times\bR: v(t,x)=g(x) \}.
\end{equation}
We now state some mild assumptions for the optimal stopping problem \eqref{eq:v}.

\begin{assumptions}\label{ass:Standard}
	The reward function $g:\bR\to\bR$ is upper semi-continuous, the value function $v:[0,T]\times\bR\to\bR$ is continuous and we have that, for every $(t,x)\in[0,T]\times\bR$,
	\begin{equation}\label{eq:IntCond}
	\bE\bigg[\sup_{s\in[t,T]} \big|g(X^{t,x}_{s})\big|\bigg]<\infty.
	\end{equation}
\end{assumptions}

Under Assumption \ref{ass:Standard}, we obtain the next three propositions, which are standard results in optimal stopping.

\begin{proposition}\label{prop:OptiStopTime}
	Let $(t,x)\in[0,T]\times\bR$, then the stopping time
	\begin{equation}\label{eq:tau*}
	\tau^*=\tau^*_{t,x}:=\inf\{s\in[0,T-t]: (t+s,X^{t,x}_{t+s})\notin \cC\}
	\end{equation}
	is optimal for the stopping problem \eqref{eq:v}.
\end{proposition}
\begin{proof}
	See, e.g., \cite[Corollary 2.9]{peskir2006optimal}.
\end{proof}

\begin{proposition}\label{prop:Mart}
	Let $(t,x)\in[0,T]\times\bR$, then the process $V:=(V_s)_{s\in[0,T-t]}$, defined by
	$$V_s=V_s^{t,x}:=v(t+s,X^{t,x}_{t+s}),$$
	is a right-continuous supermartingale and the stopped process $V^*:=(V_{s\wedge\tau^*})_{s\in[0,T-t]}$ is a right-continuous martingale.
\end{proposition}
\begin{proof}
	See, e.g., \cite[Theorem 2.4]{peskir2006optimal}.
\end{proof}

Let $\partial_t$, $\partial_x$ and $\partial_{xx}$ denote the time derivative, the spatial derivative and the second spatial derivative, respectively, and let $\partial \cC$ denote the boundary of $\cC$.

\begin{proposition}\label{prop:FreeBound}
	We have that $v\in C^{1,2}(\cC)$ and it solves the free-boundary problem
	\begin{align}\label{eq:FreeBoundProb}
	\Big(\partial_t+\mu(t,x)\partial_x +\tfrac{1}{2}(\sigma(x))^2\partial_{xx}\Big)v(t,x)&=0, \quad \qquad \quad  (t,x)\in\cC, \\
	v(t,x)&=g(x), \hspace{30pt} (t,x)\in\partial\cC,\nonumber
	\end{align}
\end{proposition}

\begin{proof}
	By Assumption \ref{ass:Standard}, $\cC$ is an open set. Then the free-boundary problem \eqref{eq:FreeBoundProb} follows, e.g., by the same arguments as in the proof of \cite[Proposition 2.6]{jacka1991optimal}.
\end{proof}

\begin{remark}
	Continuity of $v$ is not necessary to obtain Proposition \ref{prop:OptiStopTime} and Proposition \ref{prop:Mart} but lower semi-continuity would be sufficient. Moreover, these two propositions may hold with no continuity assumption on $v$: they still hold if, e.g., $g$ is continuous and non-negative and the integral condition \eqref{eq:IntCond} is satisfied (see, e.g., \cite[Appendix D]{karatzas1998methods}). For the sake of simplicity, we assume continuity of $v$,  which is necessary for Proposition \ref{prop:FreeBound}.
\end{remark}

To avoid further initial conditions on the data of the problem, we also introduce the following assumption.

\begin{assumptions}\label{ass:StopBound}
	There exists a (lower) optimal stopping boundary for the problem \eqref{eq:v}, i.e., a function $b:[0,T]\to\bR$ that separates $\cC$ from $\cD$. That is, we have
	$$\cC=\{(t,x)\in[0,T)\times\bR: x>b(t) \}$$
	and
	$$\cD=\{(t,x)\in[0,T)\times\bR: x\leq b(t) \}\cup \{T \}\times\bR.$$
\end{assumptions}
Assumption \ref{ass:StopBound} is usually proved by probabilistic arguments on a case-by-case basis (see, e.g., \cite[Proposition 2.1]{jacka1991optimal}). It is easy to see that it holds if, e.g., $x\mapsto v(t,x)-g(x)$ is non-decreasing. More general sufficient conditions that guarantee the existence of an optimal stopping boundary are shown in, e.g., \cite[Theorem 4.3]{jacka1992finite} and will be used later in Section \ref{sect:Extension}. In this paper, we prove our results when a lower stopping boundary exists but it is clear that analogous arguments would follow when an upper stopping boundary exists instead.

\section{A comparison principle}\label{sect:CompPrinc}

In this section we provide a comparison principle between solutions of the SDE \eqref{eq:X} starting at different times, which will be applied in Section \ref{sect:MainResult} to obtain monotonicity of the optimal stopping boundary.

We denote by $\cS\subseteq\bR$ the state space of the process $X$ defined in \eqref{eq:X}. For every $(t,x)\in[0,T]\times\bR$, and for a non-empty set $\cO\subseteq[0,T]\times\cS$, we define
$$\tau_\cO=\tau_\cO^{t,x}:=\inf\{s\geq0: (t+s,X^{t,x}_{t+s})\notin \cO \}\wedge(T-t).$$

\begin{lemma}\label{lemma:CompPrinc}
	Let $(t,x)\in [0,T]\times\bR$ and let $\cO\subseteq[0,T]\times\cS$ be non-empty. Assume that
	\begin{equation}\label{eq:bAssump}
	\mu(s,y)\leq \mu(u,y), \qquad \forall \: (s,y)\in \cO, \quad \forall \: u\in[0,s].
	\end{equation}
	Then, for every $u\in[0,t]$, we have that
	\begin{equation}
	\bP\Big(X^{t,x}_{t+s\wedge\tau_\cO}\leq X^{u,x}_{u+s\wedge\tau_\cO}, \quad \forall\: s\in[0,T-t]\Big)=1,
	\end{equation}
	where $\tau_\cO=\tau_\cO^{t,x}$.
\end{lemma}
\begin{proof}
	Let $X^{1}_s:=X^{t,x}_{t+s}$, $X^2_s:=X^{u,x}_{u+s}$, $\mu_1(s,y):=\mu(t+s,y)$ and $\mu_2(s,y):=\mu(u+s,y)$ with $u\in[0,t]$. Thus, for $i=1,2$, we have
	$$X^i_s = x+\int_0^s \mu_i(r,X^i_r)\ud r+\int_0^s \sigma(X^i_r)\ud W_r.$$
	Then, for $Y:=X^1-X^2$ by assumption \eqref{eq:Assh}, we obtain
	$$\int_0^s h(Y_r)^{-1}\ind_{\{Y_r>0 \}}\ud[Y]_r=\int_0^s h(|X^1_r-X^2_r|)^{-1}\big(\sigma(X^1_r)-\sigma(X^2_r)\big)^2\ind_{\{Y_r>0 \}}\ud r\leq s.$$
	Therefore, we have (see, e.g., \cite[Ch.~V, Prop.~39.3]{rogers2000diffusions}) that $L^0_s(Y)=0$ for every $s\in[0,T]$, where $L^0(Y)$ denotes the local time of $Y$ at $0$. Thus, by Tanaka's formula, for every $s\in[0,T-t]$ we obtain
	\begin{align*}
	\big(X^1_{s\wedge\tau_\cO}-X^2_{s\wedge\tau_\cO}\big)^+&=\int_{0}^{s\wedge\tau_\cO}\big(\mu_1(r,X^1_r)-\mu_2(r,X^2_r)\big)\ind_{\{X^1_r-X^2_r>0 \}}\ud r\\
	&\hspace{12pt}+\int_{0}^{s\wedge\tau_\cO}\big(\sigma(X^1_r)-\sigma(X^2_r)\big)\ind_{\{X^1_r-X^2_r>0 \}}\ud W_r,
	\end{align*}
	where $\tau_\cO=\tau_\cO^{t,x}$ and $(x)^+:=\max\{x,0 \}$. Hence,
	\begin{align*}
	0&\leq \bE\Big[\big(X^1_{s\wedge\tau_\cO}-X^2_{s\wedge\tau_\cO}\big)^+\Big]\\
	&=\bE\Big[\int_{0}^{s\wedge\tau_\cO}\big(\mu(t+r,X^1_r)-\mu(u+r,X^2_r)\big)\ind_{\{X^1_r-X^2_r>0 \}}\ud r\Big]\\
	&\leq \bE\Big[\int_{0}^{s\wedge\tau_\cO}\big(\mu(u+r,X^1_r)-\mu(u+r,X^2_r)\big)\ind_{\{X^1_r-X^2_r>0 \}}\ud r\Big]\\
	&\leq \bE\Big[\int_{0}^{s\wedge\tau_\cO} K(u+r)\big(X^1_{r}-X^2_{r}\big)^+\ud r\Big],
	\end{align*}
	where $K(t)>0$ is the Lipschitz constant for $x\mapsto \mu(t,x)$ and the second to last inequality follows from assumption \eqref{eq:bAssump}. Then, by Gronwall's lemma, we obtain that
	$$\bE\Big[\big(X^1_{s\wedge\tau_\cO}-X^2_{s\wedge\tau_\cO}\big)^+\Big]=0, \quad \forall \: s\in[0,T-t],$$
	and by continuity of $Y=X^1-X^2$ we reach the desired result.
\end{proof}

\begin{remark}
	Let $(t,x)\in[0,T]\times\bR$. Notice that if $\cO=[0,T]\times\cS$, then $\tau^{t,x}_\cO=T-t$  and so the result of Lemma \ref{lemma:CompPrinc} reads
	\begin{equation}
	\bP\Big(X^{t,x}_{t+s}\leq X^{u,x}_{u+s}, \quad \forall\: s\in[0,T-t]\Big)=1.
	\end{equation}
\end{remark}

\section{Main results}\label{sect:MainResult}

In this section we illustrate our main result for the optimal stopping problem \eqref{eq:v}, which provides monotonicity of $t\mapsto v(t,x)$ and in turn implies monotonicity of the stopping boundary. This is obtained by means of Lemma \ref{lemma:CompPrinc} and will be presented in two versions (Theorem \ref{thm:v_t} and Theorem \ref{thm:v_tTauB}) under different assumptions.

\begin{theorem}\label{thm:v_t}
	Let Assumption \ref{ass:Standard} hold. Moreover, assume that
	\begin{enumerate}[(i)]
		\item $x\mapsto g(x)$ is non-decreasing.
		\item $t\mapsto \mu(t,x)$ is non-increasing for every $x\in\cS$.
	\end{enumerate}
	Then, $t\mapsto v(t,x)$ is non-increasing for every $x\in\bR$. 
\end{theorem}

\begin{proof}
	Let $(t,x)\in[0,T]\times\bR$ and $u\in[0,t]$. By assumption (ii), we can apply Lemma \ref{lemma:CompPrinc} with $\cO=[0,T]\times\cS$ and obtain that
	\begin{equation}\label{eq:ComparisonFull}
	\bP\Big(X^{t,x}_{t+s}\leq X^{u,x}_{u+s}, \quad \forall\: s\in[0,T-t]\Big)=1.
	\end{equation}
	By the (super)martingale property of the value function (recall Proposition \ref{prop:Mart}) and since $\tau^*=\tau^*_{t,x}$ is optimal for $v(t,x)$ and sub-optimal for $v(u,x)$, we have that
	\begin{align}\label{eq:vDiff}
	v(t,x)-v(u,x)&=V_0^{t,x}-V_0^{u,x}\leq \bE\Big[V_{\tau^*}^{t,x}-V_{\tau^*}^{u,x}\Big]\nonumber\\
	&=\bE\Big[v(t+\tau^*,X^{t,x}_{t+\tau^*})-v(u+\tau^*,X^{u,x}_{u+\tau^*}) \Big] \nonumber\\
	&\leq\bE\Big[g(X^{t,x}_{t+\tau^*})-g(X^{u,x}_{u+\tau^*}) \Big]
	\leq 0,
	\end{align}
	where to obtain the last inequality we have used assumption (i) and result \eqref{eq:ComparisonFull}. Hence, $t\mapsto v(t,x)$ is non-increasing for every $x\in\bR$.
\end{proof}

We now show that we can weaken the monotonicity assumption on $t\mapsto \mu(t,x)$ but still obtain, under an additional assumption, the same result as in Theorem \ref{thm:v_t}. This alternative version partially relies on the free-boundary problem \eqref{eq:FreeBoundProb} and turns out to be useful in some optimal stopping problems, as we will illustrate in Section \ref{sect:Examples}.

Let $\cM:=\{(t,x)\in[0,T]\times\cS:\mu(t,x)<0 \}$ and let us denote by $\cM^c$ its complement, i.e.,
\begin{equation}\label{eq:Mc}
\cM^c:=([0,T]\times\cS)\setminus \cM=\{(t,x)\in[0,T]\times\cS: \mu(t,x)\geq 0 \}.
\end{equation}
Throughout this paper we also assume that $\cM$ is an open set, so that  $(t+\tau_\cM,X_{t+\tau_\cM})\in \cM^c$ on $\{\tau_\cM<T-t \}$, where recall that
\begin{equation}\label{eq:TauM}
\tau_\cM=\tau_\cM^{t,x}:=\inf\{s\geq0: (t+s,X^{t,x}_{t+s})\notin \cM \}\wedge(T-t).
\end{equation}
This holds if, e.g., $\mu$ is upper semi-continuous.

\begin{theorem}\label{thm:v_tTauB}
	 Let Assumption \ref{ass:Standard} hold. Moreover, assume that
	\begin{enumerate}[(i)]
		\item $x\mapsto g(x)$ is non-decreasing.
		\item $\mu(t,x)\leq \mu(t-\eps,x)$ for every $(t,x)\in \cM$, $\eps\in(0,t)$.
		\item $ \sigma^2(x)\partial_{xx} v(t,x)\geq -2\mu(t,x)\partial_x v(t,x)$ for every $(t,x)\in\cC\cap \cM^c$.
	\end{enumerate}
	Then, $t\mapsto v(t,x)$ is non-increasing for every $x\in\bR$.
\end{theorem}

\begin{remark}\label{Rmk:Assumpt(i)}
	Since $x\mapsto X^{t,x}$ is non-decreasing (see, e.g., \cite[Ch.~V, Th.~43.1]{rogers2000diffusions}) and, under the assumptions of Theorem \ref{thm:v_tTauB}, $x\mapsto g(x)$ is non-decreasing, we also have that $x\mapsto v(t,x)$ is non-decreasing. Thus, notice that assumption (iii) holds, in particular, if $x\mapsto v(t,x)$ is convex. This is in turn implied by convexity of $x\mapsto X^{t,x}$ and of $x\mapsto g(x)$. Therefore, assumption (iii) of Theorem \ref{thm:v_tTauB} can be substituted by convexity of $x\mapsto X^{t,x}$ and of $x\mapsto g(x)$. However, if $\sigma(x)$ is sufficiently small or if $\mu(t,x)\partial_x v(t,x)$ is sufficiently large on $\cM^c$, then we may not need $x\mapsto v(t,x)$ to be convex in order to satisfy assumption (iii).
\end{remark}

\begin{proof}
	We prove the result of the theorem in two steps. We first show that $\partial_t v(t,x)\leq 0$ for every $(t,x)\notin\partial\cC$ and we then prove that this implies that $t\mapsto v(t,x)$ is non-increasing for every $x\in\bR$.
	
	\textit{Step 1.} If $(t,x)\in\cD\setminus\partial\cC$ then $v(t,x)=g(x)$ and so $\partial_t v(t,x)= 0$. If $(t,x)\in\cC\cap \cM^c$ (we can skip this step if $\cM^c=\emptyset$), then by \eqref{eq:FreeBoundProb}
	$$\partial_t v(t,x)+\mu(t,x)\partial_x v(t,x) +\tfrac{1}{2}(\sigma(x))^2\partial_{xx}v(t,x)=0,$$
	and, by assumption (iii), we obtain
	\begin{equation}\label{eq:v_t}
	\partial_t v(t,x)\leq 0, \qquad \forall \: (t,x)\in\cC\cap \cM^c.
	\end{equation}
	
	To conclude the proof we consider $(t,x)\in\cC\cap \cM$ (we can skip this step if $\cM=\emptyset$). By assumption (ii) we can apply Lemma \ref{lemma:CompPrinc} with $\cO=\cM$ and, for every $\eps\in(0,t)$, we obtain
	\begin{equation}\label{eq:Comparison}
	\bP\Big(X^{t,x}_{t+s\wedge\tau_\cM}\leq X^{t-\eps,x}_{t-\eps+s\wedge\tau_\cM}, \quad \forall\: s\in[0,T-t]\Big)=1,
	\end{equation}
	where $\tau_\cM=\tau_\cM^{t,x}$ is defined in \eqref{eq:TauM}. Let $\eps\in(0,t)$, $\tau^*=\tau^*_{t,x}$ (recall \eqref{eq:tau*}) and $\rho:=\tau^*\wedge\tau_\cM$. By the (super)martingale property of the value function (recall Proposition \ref{prop:Mart}) and since $\tau^*$ is optimal for $v(t,x)$ and $\rho$ is sub-optimal for $v(t-\eps,x)$, we have that
	\begin{align}\label{eq:DiffQuot}
	v(t,x)-v(t-\eps,x)&\leq\bE\Big[v(t+\rho,X^{t,x}_{t+\rho})-v(t-\eps+\rho,X^{t-\eps,x}_{t-\eps+\rho}) \Big]\nonumber \\
	&\leq\bE\Big[\ind_{\{\tau^*\leq \tau_\cM \}}\big(g(X^{t,x}_{t+\tau^*})-g(X^{t-\eps,x}_{t-\eps+\tau^*}) \big)\Big]\nonumber \\
	&\hspace{12pt}+\bE\Big[\ind_{\{\tau_\cM<\tau^* \}}\big(v(t+\tau_\cM,X^{t,x}_{t+\tau_\cM})-v(t-\eps+\tau_\cM,X^{t-\eps,x}_{t-\eps+\tau_\cM}) \big)\Big]\nonumber \\
	&=\bE\Big[\ind_{\{\tau^*\leq \tau_\cM \}}\big(g(X^{t,x}_{t+\tau^*})-g(X^{t-\eps,x}_{t-\eps+\tau^*}) \big)\Big]\nonumber \\
	&\hspace{12pt}+\bE\Big[\ind_{\{\tau_\cM<\tau^* \}}\big(v(t+\tau_\cM,X^{t,x}_{t+\tau_\cM})-v(t-\eps+\tau_\cM,X^{t,x}_{t+\tau_\cM}) \big)\Big]\nonumber\\
	&\hspace{12pt}+\bE\Big[\ind_{\{\tau_\cM<\tau^* \}}\big(v(t-\eps+\tau_\cM,X^{t,x}_{t+\tau_\cM})-v(t-\eps+\tau_\cM,X^{t-\eps,x}_{t-\eps+\tau_\cM}) \big)\Big]\nonumber\\
	&\leq \bE\Big[\ind_{\{\tau_\cM<\tau^* \}}\big(v(t+\tau_\cM,X^{t,x}_{t+\tau_\cM})-v(t-\eps+\tau_\cM,X^{t,x}_{t+\tau_\cM}) \big)\Big],
	\end{align}
	where to obtain the last inequality we have used assumption (i) and result \eqref{eq:Comparison} for the first term; result \eqref{eq:Comparison} and the fact that $x\mapsto v(t,x)$ is non-decreasing (recall Remark \ref{Rmk:Assumpt(i)}) for the third term. Dividing by $\eps$, letting $\eps\to 0$ and applying dominated convergence theorem (by assumption \eqref{eq:IntCond}), we obtain
	$$\partial_t v(t,x)\leq \bE\Big[\ind_{\{\tau_\cM<\tau^* \}}\partial_t v(t+\tau_\cM,X^{t,x}_{t+\tau_\cM})\Big]\leq 0,$$
	where the last inequality follows from \eqref{eq:v_t}. Hence, $\partial_t v(t,x)\leq 0$ also for $(t,x)\in\cC\cap \cM$ and the proof of Step 1 is completed.
	
	\textit{Step 2.} If $(t,x)\in\cD$, then $v(t,x)=g(x)$ and, since $v(s,x)\geq g(x)$ for every $(s,x)\in[0,T]\times\bR$, then $v(s,x)\geq v(t,x)$ for every $s\in[0,t]$.
	
	Now let $(t,x)\in\cC$. We want to show that also $(s,x)\in\cC$ for every $s\in[0,t]$, which by Step 1 would imply that $v(s,x)\geq v(t,x)$ for every $s\in[0,t]$ and would conclude the proof. Assume, by contradiction, that
	$$D_{t,x}:=\{s\in[0,t]: (s,x)\in\cD \}\neq\emptyset$$
	and let $t_0:=\sup D_{t,x}$. Recall that, since $(t,x)\in\cC$, we have $v(t,x)>g(x)$. Since $v$ is continuous (by Assumption \ref{ass:Standard}), then $t_0<t$ and $t_0\in D_{t,x}$, i.e., $(t_0,x)\in\cD$ and so $v(t_0,x)=g(x)$. Moreover, by definition of $t_0$, we have $(s,x)\in\cC$ for every $s\in(t_0,t]$ and so
	$$v(t,x)-v(s,x)=\int_{t_0}^t \partial_t v(s,x)\ud s\leq 0, \qquad \forall \: s\in(t_0,t],$$
	where the last inequality follows from Step 1. Hence, by continuity of $v$, we have that $v(t_0,x)\geq v(t,x)$. This leads to a contradiction, as we would obtain
	$$g(x)=v(t_0,x)\geq v(t,x)>g(x).$$
\end{proof}

\begin{remark}
	If assumption (ii) in Theorem \ref{thm:v_t} (and similarly for assumptions (ii) and (iii) in Theorem \ref{thm:v_tTauB}) is substituted by a symmetric assumption (i.e., if $t\mapsto \mu(t,x)$ is increasing) then, in infinite-horizon problems, we would obtain a symmetric result, i.e., $t\mapsto v(t,x)$ would be increasing. However, this is, in general, not the case for finite-horizon problems. In that context we would have two opposite driving effects as time increases: the drift $\mu$ that increases and the stopping time domain $\cT_t$ that shrinks. The former leads to an increase of the value function with respect to time, whereas the latter leads to a decrease of the value function with respect to time. In order to study the monotonicity of $t\mapsto v(t,x)$ in such problems, it would be necessary (and, perhaps, not sufficient) to have a quantitative information on the monotonicity of $t\mapsto \mu(t,x)$.
\end{remark}

\begin{remark}
	In some cases it is possible to apply a pure PDE approach, as in \eqref{eq:v_t}, and to derive monotonicity of $t\mapsto v(t,x)$ also when the diffusion coefficient may be time-dependent. Consider the same SDE as in \eqref{eq:X} but when also $\sigma$ may be a function of time, i.e., $\sigma:[0,T]\times\bR\to\bR$. If $\mu(t,x)\geq 0$ for every $(t,x)\in[0,T]\times\cS$, then $\cM^c=[0,T]\times\cS$ (recall \eqref{eq:Mc}). Thus, under assumptions (i) and (iii) of Theorem \ref{thm:v_tTauB} and in the same way as in \eqref{eq:v_t}, we would obtain
	$$\partial_t v(t,x)\leq 0, \qquad \forall\: (t,x)\notin\partial\cC.$$
	Monotonicity of $t\mapsto v(t,x)$ then follows as in Step 2 of the proof of Theorem \ref{thm:v_tTauB}.
\end{remark}

Monotonicity of $t\mapsto v(t,x)$, which follows from either Theorem \ref{thm:v_t} or Theorem \ref{thm:v_tTauB}, then yields monotonicity of the optimal stopping boundary.

\begin{corollary}\label{cor:Monot}
	If $t\mapsto v(t,x)$ is non-increasing for every $x\in\bR$ and Assumption \ref{ass:StopBound} holds, then the optimal stopping boundary $t\mapsto b(t)$ is non-decreasing.
\end{corollary}

\begin{proof}
	Let $(t,x)\in\cC$. Then, $v(t,x)>g(x)$ and, since $t\mapsto v(t,x)$ is non-increasing, we obtain that $v(s,x)\geq v(t,x)>g(x)$ for every $s\in[0,t]$. Hence, also $(s,x)\in\cC$ and thus $t\mapsto b(t)$ is non-decreasing.
\end{proof}

\begin{remark}
	Monotonicity of $t\mapsto v(t,x)$ is also a helpful result to obtain continuity of the stopping boundary (see, e.g., arguments as in \cite[Section 3]{de2015note} and \cite[Lemma 4]{de2020optimal}).
\end{remark}

\section{Extension to time-dependent reward functions}\label{sect:Extension}

In this section we show how to obtain monotonicity of the stopping boundary for more general time-inhomogeneous optimal stopping problems, which include a running reward function and a terminal reward function that may also depend on time. We consider the same underlying framework of Section \ref{sect:PbForm} but we study the optimal stopping problem
\begin{equation}\label{eq:vExt}
v(t,x):=\sup_{\tau\in\cT_t}\bE\bigg[\int_0^{\tau} f(t+s,X^{t,x}_{t+s})\ud s+g(t+\tau,X^{t,x}_{t+\tau})\bigg], \qquad (t,x)\in[0,T]\times\bR,
\end{equation}
where $X=X^{t,x}$ is defined in \eqref{eq:X}, $f:[0,T]\times\bR\to\bR$ is a running reward function and $g:[0,T]\times\bR\to\bR$ is a terminal reward function. For this problem the continuation region $\cC$ and the stopping region $\cD$ are defined, respectively, by
\begin{equation}
\cC:=\{(t,x)\in[0,T]\times\bR: v(t,x)>g(t,x) \}
\end{equation}
and
\begin{equation}
\cD:=\{(t,x)\in[0,T]\times\bR: v(t,x)=g(t,x) \}.
\end{equation}
We then introduce the following assumption.
\begin{assumptions}\label{ass:StandardExt}
	We have that $g\in C^{1,2}([0,T]\times\bR)$, the value function $v:[0,T]\times\bR\to\bR$ is continuous and, for every $(t,x)\in[0,T]\times\bR$,
	$$\bE\bigg[\sup_{s\in[t,T]} \Big|\int_0^{s} f(r,X^{t,x}_{r})\ud r+g(s,X^{t,x}_{s})\Big|\bigg]<\infty.$$
\end{assumptions}
Notice that Assumption \ref{ass:StandardExt} is the analogous of Assumption \ref{ass:Standard} except for the stronger regularity of $g$. Under this regularity, we can apply Ito's formula and obtain that
$$g(t+s,X_{t+s})=g(t,x)+\int_0^s \cL g(t+r,X_{t+r})\ud r, \qquad \forall\: s\in[0,T-t],$$
where $\cL$ is defined by
$$\cL g(t,x):=\big(\partial_t+\mu(t,x)\partial_x+\tfrac 1 2 (\sigma(x))^2\partial_{xx}\big) g(t,x).$$
The function $w(t,x):=v(t,x)-g(t,x)$ is, thus, the value function for the optimal stopping problem
\begin{equation}\label{eq:w}
w(t,x)=\sup_{\tau\in\cT_t}\bE\bigg[\int_0^\tau h(t+s,X^{t,x}_{t+s})\ud s\bigg],\qquad (t,x)\in[0,T]\times\bR,
\end{equation}
where $h(t,x):=f(t,x)+\cL g(t,x)$. Then, notice that
$$\cC=\{(t,x)\in[0,T]\times\bR: w(t,x)>0 \}$$
and
$$\cD=\{(t,x)\in[0,T]\times\bR: w(t,x)=0 \}.$$
Analogously to Section \ref{sect:PbForm}, under Assumption \ref{ass:StandardExt}, we have that the stopping time
\begin{equation}
\tau^*=\tau^*_{t,x}:=\inf\{s\in[0,T-t]: (t+s,X^{t,x}_{t+s})\notin \cC\}
\end{equation}
is optimal for the problem \eqref{eq:vExt} (and thus also for the problem \eqref{eq:w}). Moreover, we obtain that the process $V:=(V_s)_{s\in[0,T-t]}$, defined by
\begin{equation}\label{eq:Vmart}
V_s=V_s^{t,x}:=\int_0^s h(t+r,X^{t,x}_{t+r})\ud r+w(t+s,X^{t,x}_{t+s}),
\end{equation}
is a right-continuous supermartingale and the process $V^*:=(V_{s\wedge\tau^*})_{s\in[0,T-t]}$ is a right-continuous martingale.

\begin{remark}
	For the sake of simplicity, we have assumed $g\in C^{1,2}([0,T]\times\bR)$ but one may consider different (weaker) conditions in order to apply Ito's formula and reformulate the stopping problem \eqref{eq:vExt} into the stopping problem \eqref{eq:w}.
\end{remark}

We then study the optimal stopping problem \eqref{eq:vExt} by means of the equivalent problem \eqref{eq:w}. We have the following result on the existence of an optimal stopping boundary.

\begin{proposition}\label{prop:BoundExist}
	Assume that $x\mapsto h(t,x)$ is non-decreasing for every $t\in[0,T]$, then there exists a lower optimal stopping boundary for the problem \eqref{eq:vExt}, i.e., a function $b:[0,T)\to\bR\cup\{\pm\infty \}$ such that
	$$\cC=\{(t,x)\in[0,T)\times\bR: x>b(t) \}$$
	and
	$$\cD=\{(t,x)\in[0,T)\times\bR: x\leq b(t) \}\cup \{T \}\times\bR.$$
\end{proposition}
\begin{proof}
	Since $x\mapsto h(t,x)$ is non-decreasing for every $t\in[0,T]$, then $x\mapsto w(t,x)$ is non-decreasing for every $t\in[0,T]$. Hence, if $(t,x_1)\in\cD$, then $(t,x_2)\in\cD$ for every $x_2\in(-\infty,x_1]$. Therefore, for $t\in[0,T)$, the function
	$$b(t):=\sup\{x\in\bR: w(t,x)=0 \}$$
	is a lower optimal stopping boundary for the stopping problem \eqref{eq:w} and, thus, also for \eqref{eq:vExt}.
\end{proof}

We can now obtain monotonicity of the optimal stopping boundary also for the more general class of time-inhomogeneous optimal stopping problems in \eqref{eq:vExt}. Recall that $\cS$ denotes the state space of $X$, and that we define $h(t,x):=f(t,x)+\cL g(t,x)$ and $$\cL g(t,x):=\big(\partial_t+\mu(t,x)\partial_x+\tfrac 1 2 (\sigma(x))^2\partial_{xx}\big) g(t,x).$$

\begin{theorem}\label{thm:MonotExt}
	Let Assumption \ref{ass:StandardExt} hold. Moreover, assume that
	\begin{enumerate}[(i)]
		\item $x\mapsto h(t,x)$ is non-decreasing for every $t\in[0,T]$ and $t\mapsto h(t,x)$ is non-increasing for every $x\in\cS$.
		\item $t\mapsto\mu(t,x)$ is non-increasing for every $x\in\cS$.
	\end{enumerate}
	Then, $t\mapsto w(t,x)$ is non-increasing for every $x\in\bR$ and so the optimal stopping boundary $t\mapsto b(t)$ is non-decreasing.
\end{theorem}
\begin{proof}
		Let $(t,x)\in[0,T]\times\bR$ and $u\in[0,t]$. By assumption (ii), we can apply Lemma \ref{lemma:CompPrinc} with $\cO=[0,T]\times\cS$ and obtain that
	\begin{equation}\label{eq:ComparisonFull2}
	\bP\Big(X^{t,x}_{t+s}\leq X^{u,x}_{u+s}, \quad \forall\: s\in[0,T-t]\Big)=1.
	\end{equation}
	By the (super)martingale property \eqref{eq:Vmart} of $V$ and since $\tau^*=\tau^*_{t,x}$ is optimal for $w(t,x)$ and sub-optimal for $w(u,x)$, we have that
	\begin{align*}
	w(t,x)-w(u,x)&=V^{t,x}_0-V^{u,x}_0\leq\bE\Big[V^{t,x}_{\tau^*}-V^{u,x}_{\tau^*}\Big] \\
	&=\bE\bigg[\int_0^{\tau^*}\Big\{h(t+s,X^{t,x}_{t+s})-h(u+s,X^{u,x}_{u+s}) \Big\}\ud s\bigg]\\
	&\hspace{13pt}+\bE\Big[w(t+\tau^*,X^{t,x}_{t+\tau^*})-w(u+\tau^*,X^{u,x}_{u+\tau^*})\Big]\\
	&\leq \bE\bigg[\int_0^{\tau^*}\Big\{h(t+s,X^{t,x}_{t+s})-h(u+s,X^{t,x}_{t+s}) \Big\}\ud s\bigg]\\
	&\hspace{13pt}+\bE\bigg[\int_0^{\tau^*}\Big\{h(u+s,X^{t,x}_{t+s})-h(u+s,X^{u,x}_{u+s}) \Big\}\ud s\bigg]\leq 0,
	\end{align*}
	where the last inequality follows from assumption (i) and result \eqref{eq:ComparisonFull2}. Hence, $t\mapsto w(t,x)$ is non-increasing for every $x\in\bR$. The monotonicity of $t\mapsto b(t)$ (whose existence is guaranteed by Proposition \ref{prop:BoundExist}) is, thus, obtained by the same arguments as in the proof of Corollary \ref{cor:Monot}.
\end{proof}
\begin{remark}
	Notice that the proof of \cite[Proposition 4.4]{jacka1992finite}, which provides monotonicity of the optimal stopping boundary, holds only if the underlying process is time-homogeneous. Our Theorem \ref{thm:MonotExt} extends that result to time-inhomogeneous optimal stopping problems, under the additional assumption that $t\mapsto\mu(t,x)$ is non-increasing for every $x\in\bR$.
\end{remark}

\section{Optimal stopping under incomplete information}\label{sect:OptStoppPartInfo}

Our methods are particularly suited to study optimal stopping problems under incomplete information. To this purpose, in this section, we provide some background material on this class of problems and in Section \ref{sect:Examples} we will look into a specific example.

The common feature of these stopping problems is a random variable whose outcome is unknown to the optimiser and which affects the drift of the underlying process and/or the payoff function. The literature is vast and diverse in this field and we cite, among others, \cite{shiryaev1967two}, \cite{decamps2005investment}, \cite{ekstrom2011optimal}, \cite{ekstrom2016optimal}, \cite{ekstrom2020optimal}, \cite{ekstrom2019american}, \cite{gapeev2012pricing}, \cite{glover2020optimally} \cite{henderson2020executive}. We focus, in particular, on models as in \cite{ekstrom2020optimal} and \cite{glover2020optimally} where a random variable affects the drift of the underlying process and, in a Bayesian formulation of the problem, only the prior distribution of the random variable is known to the optimiser. As time evolves, the information obtained from observing the underlying process is used to update the initial beliefs about the unknown random variable.

Let $(\Omega,\cF,\bP)$ be a complete probability space on which it is defined a standard Brownian motion $W:=(W_t)_{t\geq0}$ and a real-valued random variable $Y$ with (prior) probability distribution $\nu$. Let $T\in(0,\infty)$ be a finite time horizon. The underlying process $X$ evolves according to
\begin{equation}\label{eq:Xprior}
\ud X_t = h(Y)\ud t+\ud W_t, \qquad X_0=x\in\bR,
\end{equation}
where $h$ is a measurable function such that $\int_\bR |h(y)|\nu(\ud y)<\infty$. Given a reward function $g:[0,T]\times\bR\to\bR$, we can then define the stopping problem
\begin{equation}\label{eq:V}
V := \sup_{\tau\in[0,T]}\bE\Big[g(\tau,X_\tau)\Big],
\end{equation}
where $\tau$ is a stopping time with respect to $(\cF^X_t)_{t\in[0,T]}$, the augmented filtration generated by $X$.

It is well-known from filtering theory (see, e.g., \cite[Proposition 3.16]{bain2009fundamentals}) that
$\bE[h(Y)|\cF^X_t]=\bE[h(Y)|X_t]=:f(t,X_t)$ where
\begin{equation}\label{eq:f}
f(t,x):=\frac{\int_\bR h(y)\e^{xy-y^2t/2}\nu(\ud y)}{\int_\bR\e^{xy-y^2t/2}\nu(\ud y)}.
\end{equation}
Moreover,
\begin{equation}\label{eq:Xf}
 \ud X_s = f(s,X_{s})\ud s+\ud B_s,
\end{equation}
where $B_t:=\int_0^t \big(h(Y)-\bE[h(Y)|\cF^X_s]\big)\ud s+W_t$ is an $\cF^X$-Brownian motion (see, e.g., \cite[Proposition 2.30]{bain2009fundamentals}) known as the ``innovation process''. By means of \eqref{eq:Xf}, we can embed the original stopping problem \eqref{eq:V} into a Markovian framework and define the value function
\begin{equation}\label{eq:vIncInf}
v(t,x):=\sup_{\tau\in[0,T-t]}\bE\Big[g(t+\tau,X^{t,x}_{t+\tau})\Big],
\end{equation}
where $\tau$ is a stopping time with respect to the augmented filtration generated by $B$, and $X^{t,x}$ follows the dynamics in \eqref{eq:Xf} with $X_t^{t,x}=x\in\bR$. The optimal stopping problem \eqref{eq:vIncInf} is now in the form of \eqref{eq:vExt} and we may, thus, apply Theorem \ref{thm:MonotExt} to obtain monotonicity of its stopping boundary. The form of the drift coefficient $f(t,x)$ defined in \eqref{eq:f} highly depends on the prior distribution $\nu$ and in Section \ref{sect:Examples} we will look at a simple example where the monotonicity required in assumption (ii) of Theorem \ref{thm:MonotExt} holds.

\section{Examples}\label{sect:Examples}

In this section we show some simple examples of time-inhomogeneous optimal stopping problems where our results apply. We consider different underlying time-inhomogeneous diffusions of the form \eqref{eq:X} that give rise to corresponding optimal stopping problems of the form \eqref{eq:v}.
We then determine under which conditions on the data of the problems we can apply our theorems and obtain monotonicity of the stopping boundary. We fix a complete probability space $(\Omega,\cF,\bP)$ with a filtration $\bF:=(\cF_t)_{t\geq 0}$ satisfying the usual conditions. We let $W:=(W_t)_{t\geq0}$ be a standard Brownian motion which is $\bF$-adapted and $T\in(0,\infty)$ be a finite time horizon. For simplicity, in the following examples we also let Assumption \ref{ass:Standard} and Assumption \ref{ass:StopBound} hold.

\subsection{Applications of Theorem \ref{thm:v_t}}
We begin by considering two simple time-inhomogeneous diffusions and the corresponding optimal stopping problems of the form \eqref{eq:v}. For the reward function of these two examples we only assume that $x\mapsto g(x)$ is non-decreasing, as required in Theorem \ref{thm:v_t}.

First, for $(t,x)\in[0,T]\times\bR$, let $X=X^{t,x}$ be a Brownian motion with time-dependent drift, described by
\begin{equation}\label{eq:BmDrift}
X_{t+s}=x+\int_0^s\mu(t+r)\ud r+\sigma W_s, \qquad s\in[0,T-t]. 
\end{equation}
If $t\mapsto\mu(t)$ is non-increasing, then we can apply Theorem \ref{thm:v_t} and obtain that $t\mapsto v(t,x)$ is non-increasing for every $x\in\bR$. Thus, Corollary \ref{cor:Monot} guarantees that the corresponding optimal stopping boundary $t\mapsto b(t)$ is non-decreasing.

Now, for $(t,x)\in[0,T]\times(0,\infty)$, let $X=X^{t,x}$ be a geometric Brownian motion with time-dependent drift, described by
\begin{equation}
X_{t+s}=x+\int_0^s\gamma(t+r)X_{t+r}\ud r+\int_0^s\sigma X_{t+r} \ud W_r, \qquad s\in[0,T-t]. 
\end{equation}
Its state space is $\cS=(0,\infty)$ and its drift is $\mu(t,x):=x\gamma(t)$. If $t\mapsto\gamma(t)$ is non-increasing, we can apply Theorem \ref{thm:v_t} and obtain that $t\mapsto v(t,x)$ is non-increasing for every $x\in\bR$. Thus, Corollary \ref{cor:Monot} guarantees that the corresponding optimal stopping boundary $t\mapsto b(t)$ is non-decreasing.

Notice that the assumptions on the monotonicity of the drifts for the previous two examples could be weakened by considering assumption (ii) of Theorem \ref{thm:v_tTauB} instead. However, this comes with the cost of adding assumption (iii) of Theorem \ref{thm:v_tTauB} (which is implied by convexity of $x\mapsto X^{t,x}$ and of $x\mapsto g(x)$, recall Remark \ref{Rmk:Assumpt(i)}).

\subsection{Applications of Theorem \ref{thm:v_tTauB}}
For some underlying time-inhomogeneous diffusion it happens that assumption (ii) of Theorem \ref{thm:v_t} does not hold. In these cases it is useful to consider Theorem \ref{thm:v_tTauB}, which weakens this assumption. We now show an example of this situation. This application is even more suited to our techniques as no specific assumption on the drift of the underlying process is required in order to apply Theorem \ref{thm:v_tTauB}. This is the case when the underlying time-inhomogeneous diffusion is a Brownian bridge and an example of this optimal stopping problem, where $g(x)=\e^x$, is studied in \cite{de2020optimal}. For our example we assume that $x\mapsto g(x)$ is non-decreasing (as required by assumption (i) of Theorem \ref{thm:v_tTauB}) and, for simplicity, that it is convex (to satisfy assumption (iii) of Theorem \ref{thm:v_tTauB}, but convexity is not strictly necessary, as explained in Remark \ref{Rmk:Assumpt(i)}).

For $(t,x)\in[0,T]\times\bR$, let $X=X^{t,x}$ be a Brownian bridge pinned at $0$ at time $T\in(t,\infty)$, whose dynamics are described by
\begin{equation}\label{eq:BrownBridge}
X_{t+s}=x-\int_0^s \frac{X_{t+r}}{T-t-r}\ud r+\sigma W_s, \qquad s\in[0,T-t), 
\end{equation}
with $X_T=0$. It is easy to check that the unique strong solution to this SDE is given by
$$X_{t+s}=(1-t-s)\bigg(\frac{x}{1-t}+\int_0^s \frac{1}{1-t-r}\ud W_r\bigg), \qquad s\in[0,T-t).$$
Hence, $x\mapsto X^{t,x}$ is linear (and thus convex) and together with convexity of $x\mapsto g(x)$ we have that assumption (iii) of Theorem \ref{thm:v_tTauB} is satisfied (recall Remark \ref{Rmk:Assumpt(i)}). In order to apply Theorem \ref{thm:v_tTauB}, we are only left to check that assumption (ii) is satisfied. The drift coefficient of the SDE \eqref{eq:BrownBridge} is $\mu(t,x):=-x/(T-t)$. Thus, we have that
$$\cM:=\{(t,x)\in[0,T)\times\bR: \mu(t,x)<0 \}=\{(t,x)\in[0,T)\times\bR: x>0 \}$$
and
$$\mu(t,x)=-\frac{x}{T-t}\leq - \frac{x}{T-t+\eps}=\mu(t-\eps,x), \qquad (t,x)\in \cM, \quad  \eps\in(0,t).$$
Therefore, also assumption (ii) of Theorem \ref{thm:v_tTauB} holds.
Hence, we can apply Theorem \ref{thm:v_tTauB} and obtain that $t\mapsto v(t,x)$ is non-increasing for every $x\in\bR$. Thus, Corollary \ref{cor:Monot} guarantees that the corresponding optimal stopping boundary $t\mapsto b(t)$ is non-decreasing.

Further examples can be shown to satisfy the assumptions of Theorem \ref{thm:v_tTauB}. For instance, optimal stopping problems where the underlying diffusion follows the Vasicek model (i.e., an Ornstein-Ulhenbeck process) or the Cox-Ingersoll-Ross model can be studied in their time-inhomogeneous version, i.e., when the long-term mean is allowed to be time-dependent.

\subsection{An application to optimal stopping under incomplete information}
To conclude, we consider an example of an optimal stopping problem under incomplete information as in Section \ref{sect:OptStoppPartInfo}, for which we want to apply Theorem \ref{thm:v_t}. Assume that $X$ follows the dynamics as in \eqref{eq:Xprior} with $h(y)=y$, i.e.,
$$\ud X_t = Y\ud t+\ud W_t,$$
where $Y$ is a real-valued random variable with prior distribution $\nu$ which has finite first moment. Then, by filtering theory, we have that
$$\ud X_t=f(t,X_t)\ud t+\ud B_t,$$
where
$$f(t,x):=\frac{\int_\bR y\e^{xy-y^2t/2}\nu(\ud y)}{\int_\bR\e^{xy-y^2t/2}\nu(\ud y)},$$
and $B$ is a Brownian motion with respect to the filtration generated by $X$. We can then consider corresponding optimal stopping problems as in \eqref{eq:vIncInf}, where the reward function $x\mapsto g(x)$ is non-decreasing (so that it satisfies assumption (i) of Theorem \ref{thm:v_t}). We now want to look at an example where also assumption (ii) of Theorem \ref{thm:v_t} is satisfied, i.e., where the drift $t\mapsto f(t,x)$ is non-increasing, so that we can obtain monotonicity of the optimal stopping boundary. Let us consider a simple example of a two-point prior distribution $\nu=p\delta_l+(1-p)\delta_r$, where $\delta_x$ is the Dirac delta centred in $x\in\bR$, $-\infty<l<r<\infty$ and $p\in(0,1)$. Then,
$$f(t,x)=\frac{p l \e^{lx-l^2t/2}+(1-p)r\e^{rx-r^2t/2}}{p \e^{lx-l^2t/2}+(1-p)\e^{rx-r^2t/2}}.$$
For assumption (ii) of Theorem to hold, we would like to obtain $\partial_t f\leq 0$. After some algebra, we have that
$$\partial_t f(t,x)=-\frac{1}{2}\frac{p(1-p)\e^{(l+r)x-(l^2+r^2)t/2}(r-l)^2(r+l)}{\big[ p \e^{lx-l^2t/2}+(1-p)\e^{rx-r^2t/2}\big]^2}.$$
If $r\geq -l$, then $\partial_t f(t,x)\leq 0$ for every $(t,x)\in[0,T]\times\bR$. Therefore, we can apply Theorem \ref{thm:v_t}, and obtain that $t\mapsto v(t,x)$ is non-increasing for every $x\in\bR$. Thus, Corollary \ref{cor:Monot} guarantees monotonicity of the corresponding stopping boundary also for this example of optimal stopping problem under incomplete information. Notice that this is only one example of prior distribution and different priors may be investigated. For instance, Theorem \ref{thm:v_t} can also be applied to the optimal stopping of a Brownian bridge with unknown pinning point whose prior is normal $\mu\sim\cN(m,\gamma^2)$ (see, \cite[Section 5.1]{ekstrom2020optimal}).

\section*{Acknowledgments}
I wish to thank S.~Villeneuve for encouraging me to pursue this research idea, and T.~De Angelis and E.~Ekstr\"{o}m for the enlightening discussions.

\bibliography{bibfile}{}

\begin{thebibliography}{10}

\bibitem{bain2009fundamentals}
A.~Bain and D.~Crisan.
\newblock {\em Fundamentals of stochastic filtering}.
\newblock Springer, 2009.

\bibitem{chen2007mathematical}
X.~Chen and J.~Chadam.
\newblock A mathematical analysis of the optimal exercise boundary for
  {A}merican put options.
\newblock {\em SIAM Journal on Mathematical Analysis}, 38(5):1613--1641, 2007.

\bibitem{d2020discounted}
B.~D'Auria, E.~Garc{\'\i}a-Portugu{\'e}s, and A.~Guada.
\newblock Discounted optimal stopping of a {B}rownian bridge, with application
  to {A}merican options under pinning.
\newblock {\em Mathematics}, 8(7):1159, 2020.

\bibitem{de2015note}
T.~De~Angelis.
\newblock A note on the continuity of free-boundaries in finite-horizon optimal
  stopping problems for one-dimensional diffusions.
\newblock {\em SIAM Journal on Control and Optimization}, 53(1):167--184, 2015.

\bibitem{de2022stopping}
T.~De~Angelis.
\newblock Stopping spikes, continuation bays and other features of optimal
  stopping with finite-time horizon.
\newblock {\em Electronic Journal of Probability}, 27:1--41, 2022.

\bibitem{de2020optimal}
T.~De~Angelis and A.~Milazzo.
\newblock Optimal stopping for the exponential of a {B}rownian bridge.
\newblock {\em Journal of Applied Probability}, 57(1):361--384, 2020.

\bibitem{de2020global}
T.~De~Angelis and G.~Peskir.
\newblock Global ${C}^{1}$ regularity of the value function in optimal stopping
  problems.
\newblock {\em The Annals of Applied Probability}, 30(3):1007--1031, 2020.

\bibitem{decamps2005investment}
J.-P. D{\'e}camps, T.~Mariotti, and S.~Villeneuve.
\newblock Investment timing under incomplete information.
\newblock {\em Mathematics of Operations Research}, 30(2):472--500, 2005.

\bibitem{ekstrom2004convexity}
E.~Ekstr{\"o}m.
\newblock Convexity of the optimal stopping boundary for the {A}merican put
  option.
\newblock {\em Journal of {M}athematical {A}nalysis and {A}pplications},
  299(1):147--156, 2004.

\bibitem{ekstrom2011optimal}
E.~Ekstr{\"o}m and B.~Lu.
\newblock Optimal selling of an asset under incomplete information.
\newblock {\em International Journal of Stochastic Analysis}, 2011, 2011.

\bibitem{ekstrom2016optimal}
E.~Ekstrom and J.~Vaicenavicius.
\newblock Optimal liquidation of an asset under drift uncertainty.
\newblock {\em SIAM Journal on Financial Mathematics}, 7(1):357--381, 2016.

\bibitem{ekstrom2020optimal}
E.~Ekstr{\"o}m and J.~Vaicenavicius.
\newblock Optimal stopping of a {B}rownian bridge with an unknown pinning
  point.
\newblock {\em Stochastic Processes and their Applications}, 130(2):806--823,
  2020.

\bibitem{ekstrom2019american}
E.~Ekstr{\"o}m and M.~Vannest{\aa}l.
\newblock American options and incomplete information.
\newblock {\em International Journal of Theoretical and Applied Finance},
  22(06):1950035, 2019.

\bibitem{ekstrom2009optimal}
E.~Ekstr{\"o}m and H.~Wanntorp.
\newblock Optimal stopping of a {B}rownian bridge.
\newblock {\em Journal of Applied Probability}, 46(1):170--180, 2009.

\bibitem{follmer1972}
H.~F{\"o}llmer.
\newblock Optimal stopping of constrained {Brownian} motion.
\newblock {\em Journal of Applied Probability}, 9(3):557--571, 1972.

\bibitem{gapeev2012pricing}
P.~V. Gapeev.
\newblock Pricing of perpetual {A}merican options in a model with partial
  information.
\newblock {\em International Journal of Theoretical and Applied Finance},
  15(01):1250010, 2012.

\bibitem{glover2020optimally}
K.~Glover.
\newblock Optimally stopping a {B}rownian bridge with an unknown pinning time:
  a {B}ayesian approach.
\newblock {\em Stochastic Processes and their Applications}, 2020.

\bibitem{henderson2020executive}
V.~Henderson, K.~Klad{\'\i}vko, M.~Monoyios, and C.~Reisinger.
\newblock Executive stock option exercise with full and partial information on
  a drift change point.
\newblock {\em SIAM Journal on Financial Mathematics}, 11(4):1007--1062, 2020.

\bibitem{hobson2021shape}
D.~Hobson.
\newblock The shape of the value function under {P}oisson optimal stopping.
\newblock {\em Stochastic Processes and their Applications}, 133:229--246,
  2021.

\bibitem{jacka1991optimal}
S.~D. Jacka.
\newblock Optimal stopping and the {A}merican put.
\newblock {\em Mathematical Finance}, 1(2):1--14, 1991.

\bibitem{jacka1992finite}
S.~D. Jacka and J.~R. Lynn.
\newblock Finite-horizon optimal stopping, obstacle problems and the shape of
  the continuation region.
\newblock {\em Stochastics: An International Journal of Probability and
  Stochastic Processes}, 39(1):25--42, 1992.

\bibitem{jaillet1990variational}
P.~Jaillet, D.~Lamberton, and B.~Lapeyre.
\newblock Variational inequalities and the pricing of {A}merican options.
\newblock {\em Acta Applicandae Mathematica}, 21(3):263--289, 1990.

\bibitem{karatzas1998methods}
I.~Karatzas and S.~E. Shreve.
\newblock {\em Methods of Mathematical Finance}.
\newblock Springer, 1998.

\bibitem{laurence2009regularity}
P.~Laurence and S.~Salsa.
\newblock Regularity of the free boundary of an {A}merican option on several
  assets.
\newblock {\em Communications on Pure and Applied Mathematics: A Journal Issued
  by the Courant Institute of Mathematical Sciences}, 62(7):969--994, 2009.

\bibitem{oshima2006optimal}
Y.~Oshima.
\newblock On an optimal stopping problem of time inhomogeneous diffusion
  processes.
\newblock {\em SIAM Journal on Control and Optimization}, 45(2):565--579, 2006.

\bibitem{peskir2019continuity}
G.~Peskir.
\newblock Continuity of the optimal stopping boundary for two-dimensional
  diffusions.
\newblock {\em The Annals of Applied Probability}, 29(1):505--530, 2019.

\bibitem{peskir2006optimal}
G.~Peskir and A.~Shiryaev.
\newblock {\em Optimal stopping and free-boundary problems}.
\newblock Springer, 2006.

\bibitem{rogers2000diffusions}
L.~C.~G. Rogers and D.~Williams.
\newblock {\em Diffusions, Markov processes and martingales: Volume 2, It{\^o}
  calculus}.
\newblock Cambridge university press, 2000.

\bibitem{shepp1969}
L.~A. Shepp.
\newblock Explicit solutions to some problems of optimal stopping.
\newblock {\em The Annals of Mathematical Statistics}, 40(3):993, 1969.

\bibitem{shiryaev1967two}
A.~N. Shiryaev.
\newblock Two problems of sequential analysis.
\newblock {\em Cybernetics}, 3(2):63--69, 1967.

\bibitem{villeneuve1999exercise}
S.~Villeneuve.
\newblock Exercise regions of {A}merican options on several assets.
\newblock {\em Finance and Stochastics}, 3(3):295--322, 1999.

\bibitem{yang2014refined}
Y.~Yang.
\newblock Refined solutions of time inhomogeneous optimal stopping problem and
  zero-sum game via {D}irichlet form.
\newblock {\em Probability and Mathematical Statistics}, 34(2):253--271, 2014.

\end{thebibliography}
\bibliographystyle{abbrv} 

\end{document}